\documentclass[12pt]{amsart}

\usepackage{amssymb, amsmath}
\usepackage{graphics}
\usepackage[dvips]{graphicx}
\newtheorem{thm}{Theorem}[section]

\newtheorem{lmm}[thm]{Lemma}

\theoremstyle{definition}

\theoremstyle{remark}
\begin{document}
%
\title[Oscillatory integrals]
{Asymptotic limit of oscillatory integrals \\
with certain smooth phases}
\author{Joe Kamimoto and Toshihiro Nose}
\address{Faculty of Mathematics, Kyushu University, 
Motooka 744, Nishi-ku, Fukuoka, 819-0395, Japan} 
\email{joe@math.kyushu-u.ac.jp}
\address{
Faculty of Engineering, Fukuoka Institute of Technology, 
Wajiro-higashi 3-30-1, Higashi-ku, Fukuoka, 811-0295, Japan
}
\email{nose@fit.ac.jp}
\maketitle

\begin{abstract}      
We give an exact result about 
the asymptotic limit of an oscillatory integral
whose phase contains a certain flat term. 
Corresponding to the real analytic phase case, 
one can see an essential difference 
in the behavior of the above oscillatory integral.
\end{abstract}

\tableofcontents      

\section{Introduction}

Let us consider  the oscillatory integral:
\begin{equation*}
I_f(t;\varphi)=
\int_{{\mathbb R}^n}
e^{i t f(x)} \varphi(x)dx \quad\quad
t>0, 
\end{equation*}
where 
$f$ and $\varphi$ are real-valued $C^{\infty}$ smooth functions
defined on an open neighborhood $U$ of 
the origin in $\mathbb{R}^n$
and the support of $\varphi$ is compact and is contained in $U$.
Here, $f$ and $\varphi$ are called the {\it phase} and the {\it amplitude}
respectively. 

The oscillatry integral appears in many fields in mathematics 
and the information of its behavior as $t\to\infty$ often 
plays important roles in the respective field
(we only refer to \cite{agv88} and \cite{ste93}). 
Until now, 
many strong results about its behavior have been obtained. 
In particular, Varchenko \cite{var76} shows that
the behavior can be described by using the geometry of the
{\it Newton polyhedron} of the phase 
when the phase is real analytic and satisfies some conditions. 
Here, the Newton polyhedron is an important concept in singularity theory
(see \cite{agv88}).
Later, his result has been improved and generalized in many kinds of cases. 
To be more specific, 
the following result about the asymptotic limit
has been obtained in many cases 
(\cite{sch91}, \cite{dns05}, \cite{ikm10}, 
\cite{im11jfaa}, \cite{ot13}, \cite{ckn13},
\cite{kn13tro}, \cite{kn15tams}, \cite{kn15}, \cite{kn16}, etc.): 
\begin{equation}\label{eqn:1.2}
\lim_{t\to\infty} 
t^{1/d(f)}(\log t)^{-m(f)+1} \cdot 
I_f(t;\varphi)
=C_f(\varphi),
\end{equation}
where 
$d(f)$ and $m(f)$ are simply defined through 
the geometry of the Newton polyhedron 
of $f$ (see \cite{agv88}):
$d(f)$ is a positive number, called the {\it Newton distance} of $f$,
and $m(f)$ is contained in the set $\{1,\ldots,n\}$, 
called the {\it multiplicity} of $d(f)$. 
Moreover, $C_f(\varphi)$ is a constant, which is nonzero when
$\varphi(0)$ is positive and $\varphi$ is nonnegative on $U$.
We remark that 
the constant $C_f(\varphi)$ has been exactly computed in many cases
(\cite{sch91}, \cite{dns05}, \cite{ot13}, \cite{ckn13},
\cite{kn13tro}, \cite{kn15tams}, \cite{kn15}, \cite{kn16}, etc.).   

But, unfortunately, the above result (\ref{eqn:1.2}) 
cannot be extended to the general $C^{\infty}$ smooth case. 
The purpose of this note is to show the following theorem:
\begin{thm}
When 
$f(x_1,x_2)=x_2^q+e^{-1/|x_1|^p}$, where
$p$ is a positive real number and 
$q$ is an integer not less than 2, and the support of 
$\varphi$ is compact, we have
\begin{equation*}
\lim_{t\to\infty} 
t^{1/q}(\log t)^{1/p} \cdot 
\int_{{\mathbb R}^2}
e^{i t f(x_1,x_2)} \varphi(x_1,x_2)dx_1dx_2
=C_{q} \varphi(0,0),
\end{equation*}
where $C_q$ is a nonzero constant defined by  
\begin{equation*}
C_{q}=
\begin{cases}
&4\Gamma(1/q+1)\cdot e^{\frac{\pi}{2q}i}
\quad \mbox{ ($q$ is even);} \\
&4\Gamma(1/q+1)\cdot\cos\frac{\pi}{2q}
\quad \mbox{ ($q$ is odd).}
\end{cases}
\end{equation*}
\end{thm}
When $q=2$, 
Iosevich and Sawyer \cite{is97} have given an
estimate from the above: 
$|I_f(t;\varphi)|\leq C t^{-1/2}(\log t)^{-1/p}$,
with $C>0$. 

The above theorem implies that
equality (\ref{eqn:1.2}) does not hold in the above case 
(note that $d(f)=q$ and $m(f)=1$) and, moreover, 
the behavior of the oscillatory integral 
cannot always be determined by the information of
only Newton polyhedron of the phase
when the phase is smooth.
On the other hand,  
the above theorem
shows that for any positive number $\alpha$, 
there exists a phase whose oscillatory integral
satisfies 
$\lim_{t\to\infty} 
t^{1/d(f)}(\log t)^{\alpha} \cdot 
I_f(t;\varphi)
=C \varphi(0,0)$ with $C\neq 0$
in the two-dimensional case.

Throughout this article, 
we sometimes use the symbol:
$X:=\log t$ for brief description.
Moreover, we often use the same symbols $t_0$ and $C$ 
to express various constants 
which are independent of $t$.

\section{Behavior of an associated one-dimensional integral}

To prove Theorem~1.1,
we prepare some auxiliary lemma concerning
about an associated one-dimensional integral.
Let $\psi$ be a smooth function 
defined on $\mathbb R$ whose support 
is compact.
Let $L(t;\psi)$ be the integral defined by
\begin{equation*}
L(t;\psi)=\int_{0}^{\infty}
e^{ite^{-1/x^p}}\psi(x)dx,
\end{equation*}
where $p$ is a positive real number.
Moreover, $L(t;\psi)$ can be written as 
\begin{equation*}
L(t;\psi)=L^{(1)}(t;\psi)+L^{(2)}(t;\psi),
\end{equation*}
with
\begin{equation*}
\begin{split}
&L^{(1)}(t;\psi)=\int_0^{\frac{1}{(\log t)^{1/p}}}
e^{ite^{-1/x^p}}\psi(x)dx, \\
&L^{(2)}(t;\psi)=\int_{\frac{1}{(\log t)^{1/p}}}^{\infty}
e^{ite^{-1/x^p}}\psi(x)dx.
\end{split}
\end{equation*}
The asymptotic behaviors of
the integrals 
$L(t;\psi)$,  $L^{(1)}(t;\psi)$ and $L^{(2)}(t;\psi)$
as $t\to \infty$ 
are seen as follows.

\begin{lmm}
\begin{enumerate}
\item[{\rm (i)}] $$\lim_{t\to\infty}
 (\log t)^{1/p} \cdot L^{(1)}(t;\psi)=\psi(0).$$
\item[{\rm (ii)}]
$$
\lim_{t\to\infty}
 (\log t)^{1/p+1} \cdot L^{(2)}(t;\psi)=
\psi(0)\cdot\int_1^{\infty}\frac{e^{iw}}{w}dw.
$$
\end{enumerate}
In particular, we have
$$
\lim_{t\to\infty}
 (\log t)^{1/p} \cdot L(t;\psi)=\psi(0).
$$
\end{lmm}
\begin{proof}
We may assume that the support of $\psi$ is contained
in $(\frac{-1}{\log 2},\frac{1}{\log 2})$ from the principle of
stationary phase (see \cite{agv88}, \cite{ste93}).

\vspace{.5 em}
{\bf (i).} \quad 
By exchanging the integral variable 
$x$ by $u$ as
\begin{equation*}
x=\dfrac{1}{[X(u+1)]^{1/p}} 
\Longleftrightarrow
u=\dfrac{1}{x^{p} X}-1  \quad\quad (X:=\log t),
\end{equation*}
the integral $L^{(1)}(t;\psi)$ can be written as 
\begin{equation*}
\begin{split}
L^{(1)}(t;\psi)=
\frac{1}{(\log t)^{1/p}} \cdot
\int_0^{\infty}
e^{it^{-u}}
\frac{\psi\left(\frac{1}{[X(u+1)]^{1/p}}\right)}
{(u+1)^{1+1/p}}du.
\end{split}
\end{equation*}
Therefore, the Lebesgue convergence theorem implies
\begin{equation*}
\lim_{t\to\infty}
(\log t)^{1/p}\cdot
L^{(1)}(t;\psi)=\frac{1}{p}\psi(0)\cdot
\int_0^{\infty}
\frac{du}{(u+1)^{1+1/p}}
=\psi(0).
\end{equation*}

\vspace{.5 em}
{\bf (ii)}.\quad 
By exchanging the integral variable $x$ by $u$ as
\begin{equation*}
u=e^{-1/x^p}\Longleftrightarrow
x=\left(\frac{-1}{\log u}\right)^{1/p},
\end{equation*}
the integral  $L^{(2)}(t;\psi)$ can be written as 
\begin{equation}
L^{(2)}(t;\psi)=
\int_{1/t}^{1/2}
e^{itu}\frac{1}{u}\left(\frac{-1}{\log u}\right)^{1/p+1}
\tilde{\psi}(u)du.
\label{eqn:5.1}
\end{equation}
Here, let $\tilde{\psi}$ be the function defined on $[0,1)$
satisfying that  
$\tilde{\psi}(u):=\psi(\left(\frac{-1}{\log u}\right)^{1/p})$
for $u\in (0,1)$ and $\tilde{\psi}(0):=\psi(0)$.
Note that $\tilde{\psi}$ is continuous on $[0,1)$, 
smooth in $(0,1)$ and its support is contained
in $[0,1/2)$.
Applying integration by parts to (\ref{eqn:5.1}),
we have
\begin{equation}
L^{(2)}(t;\psi)=M^{(1)}(t)+M^{(2)}(t),
\label{eqn:5.2}
\end{equation}
with
\begin{equation*}
\begin{split}
&M^{(1)}(t)=
\left[\frac{1}{it}e^{itu}\frac{1}{u}
\left(\frac{-1}{\log u}\right)^{1/p+1}
\tilde{\psi}(u)\right]_{1/t}^{1/2},\\
&M^{(2)}(t)=
-\frac{1}{it}
\int_{1/t}^{1/2}
e^{itu}\frac{d}{du}
\left\{
\frac{1}{u}
\left(\frac{-1}{\log u}\right)^{1/p+1}
\tilde{\psi}(u)
\right\}du.
\end{split}
\end{equation*}
The behaviors of $M^{(1)}(t)$ and $M^{(2)}(t)$ as $t\to\infty$ 
can be seen as follows.

\vspace{.5 em}
(Estimate for $M^{(1)}(t)$.) \quad 

Noticing that the support of $\tilde{\psi}$ 
is contained in $[0,1/2)$, we have
\begin{eqnarray*}
&&M^{(1)}(t)=
\frac{2\tilde{\psi}(1/2)}{(\log 2)^{1/p+1}}
\cdot \frac{e^{it/2}}{it}
+
ie^i\frac{\tilde{\psi}(1/t)}{(\log t)^{1/p+1}}
=ie^i\frac{\tilde{\psi}(1/t)}{(\log t)^{1/p+1}}.
\end{eqnarray*}
Therefore 
\begin{equation}
\lim_{t\to\infty}(\log t)^{1/p+1}M^{(1)}(t)
=ie^i \tilde{\psi}(0)= ie^i \psi(0).
\label{eqn:5.3}
\end{equation}

\vspace{.5 em}

(Estimate for $M^{(2)}(t)$.) \quad 

By a simple computation, 
the integral $M^{(2)}(t)$ can be written as 
\begin{equation}
M^{(2)}(t)=
\frac{-1}{it}
\int_{1/t}^{1/2}
e^{itu}\frac{1}{u^2}
\left(\frac{-1}{\log u}\right)^{1/p+1}
a(u)du,
\label{eqn:5.4}
\end{equation}
where $a$ is a smooth function defined on $(0,1)$ defined by
\begin{eqnarray*}
&&a(u):=\\
&&\left[
-1+\left(\frac{1}{p}+1\right)
\left(\frac{-1}{\log u}\right)
\right]
\psi\left(\left(\frac{-1}{\log u}\right)^{1/p}\right)
+\frac{1}{p}
\psi'\left(
\left(\frac{-1}{\log u}\right)^{1/p}
\right)
\left(
\frac{-1}{\log u}
\right)^{1/p+1}.
\end{eqnarray*} 
Note that $a$ 
can be naturally extended to be continuous on $[0,1)$ 
and its support is contained in $[0,1/2)$.
Moreover, by exchanging the integral variable
$u$ by $v$ as
$$u=\frac{e^v}{t} 
\Longleftrightarrow
v=\log(ut),$$
(\ref{eqn:5.4}) can be rewritten as
\begin{eqnarray}
M^{(2)}(t)
&=&\frac{-1}{it}
\int_0^{\log t-\log 2}
e^{ie^v}
\left(
\frac{t}{e^v}
\right)^{2}
\left(
\frac{-1}{v-\log t}
\right)^{1/p+1} 
a\left(\frac{e^v}{t}\right)
\frac{e^v}{t}dv \nonumber\\
&=&
\frac{i}{X^{1/p+1}}
\int_0^{X-\log 2}
e^{ie^{v}}e^{-v}
\left(
\frac{1}{1-v/X}
\right)^{1/p+1} 
a\left(\frac{e^v}{t}\right)
dv.
\label{eqn:5.5}
\end{eqnarray}
Since 
the following inequality always holds:
$$
\frac{1}{1-v/X}\leq \frac{v}{\log 2}+1
\quad\quad \mbox{for $v\in [0,X-\log 2]$},
$$
the integrand in (\ref{eqn:5.5}) can be estimated as follows.
There exists a positive number $C$  such that 
\begin{equation}
\begin{split}
&\left|
e^{ie^{v}}e^{-v}
\left(
\frac{1}{1-v/X}
\right)^{1/p+1} 
a\left(\frac{e^v}{t}\right)
\right| \\
&\quad\quad\quad 
\leq C e^{-v}
\left(
\frac{v}{\log 2}+1
\right)^{1/p+1}
\quad \mbox{for $v\in (0,X-\log 2)$}.
\end{split}
\label{eqn:5.6}
\end{equation}
Since the right hand side of (\ref{eqn:5.6}) is integrable on $[0,\infty)$,
the Lebesgue convergence theorem implies that
\begin{equation}
\begin{split}
\lim_{t\to\infty}
(\log t)^{1/p+1}\cdot M^{(2)}(t)
&=i\int_0^{\infty} e^{ie^v}e^{-v}a(0)dv\\
&=-i\psi(0)\int_1^{\infty} \frac{e^{iw}}{w^2}dw.
\end{split}
\label{eqn:5.7}
\end{equation}
by exchanging the integral variable $v$ by $w$: 
$w=e^v$.

Putting (\ref{eqn:5.2}), (\ref{eqn:5.3}), (\ref{eqn:5.7})
together,  we obtain (ii) in Lemma~2.1.
Note that integration by parts implies  
$$
i\left(e^i-\int_1^{\infty}\frac{e^{iw}}{w^2}dw
\right)=\int_1^{\infty}\frac{e^{iw}}{w}dw.
$$
\end{proof}


{\it Remarks.} 

\begin{enumerate}
\item
In the above proof of (ii),
integration by parts for
$L^{(2)}(t;\psi)$ is crusial. 
Indeed, the behavior of $M^{(2)}(t)$  
can be more easily understood.
The essential difference between 
$L^{(2)}(t;\psi)$ and $M^{(2)}(t)$ is seen in the powers of $u$
(i.e., $1/u$ in (\ref{eqn:5.1}) and 
$1/u^2$ in (\ref{eqn:5.4}) respectively) and it plays 
useful roles in the above computation.  
\item
The integral in (ii) in Lemma~2.1 seems difficult to express 
its value in more clear form. 
But, by using the integrals:
$$
{\rm si}(z)=-\int_{z}^{\infty}\frac{\sin x}{x}dx, \quad
{\rm Ci}(z)=-\int_{z}^{\infty}\frac{\cos x}{x}dx, \quad
E_n(z)=\int_{n}^{\infty} \frac{e^{-zx}}{x}dx,
$$
the value of the integral can be expressed as  
$-{\rm Ci}(1)-i{\rm si}(1)=E_1(-i)$.
The above integrals are the so-called 
{\it  sine integral},
{\it cosine integral}, 
{\it exponetial integral}, respectively,  
which are some kinds of {\it error functions}.
(See, for example, \cite{tem}, p.6, p.60.)
\end{enumerate}
\section{The proof of Theorem~1.1}

We respectively define the integrals:
\begin{equation*}
\tilde{I}^{(\pm)}(t)
=
\int_0^{\infty}\int_0^{\infty}
e^{it[\pm x_2^q+e^{-1/|x_1|^p}]}
\varphi(x_1,x_2)dx_1dx_2.
\end{equation*}
The integral
$I_f(t;\varphi)$ 
can be written as 
\begin{equation*}
I_f(t;\varphi)
=
\sum_{(\theta_1,\theta_2)\in \{\pm 1,\pm 1\}}
\int_0^{\infty}\int_0^{\infty}
e^{it[\theta_2^q x_2^q+e^{-1/|x_1|^p}]}
\varphi(\theta_1 x_1,\theta_2 x_2)dx_1dx_2.
\end{equation*}
Therefore, in order to prove the theorem, 
it suffices to show 
\begin{equation}
\lim_{t\to\infty}
t^{1/q}(\log t)^{1/p} \cdot \tilde{I}^{(\pm)}(t)
=\Gamma(1/q+1)\cdot e^{\pm\frac{\pi}{2q}i}\cdot 
\varphi(0,0).
\label{eqn:5.8}
\end{equation}
Since the form of $\tilde{I}^{(-)}(t)$ is similar 
to that of $\tilde{I}^{(+)}(t)$,
we only consider the case of the integral 
$\tilde{I}^{(+)}(t)$.

Now, the integral $\tilde{I}^{(+)}(t)$ can be 
devided as follows. 
\begin{equation}
\tilde{I}^{(+)}(t)=
J^{(1)}(t)+J^{(2)}(t),
\label{eqn:5.9}
\end{equation}
with
\begin{equation}
\begin{split}
&J^{(1)}(t)
=
\int_0^{\infty}\int_0^{\frac{1}{(\log t)^{1/p}}}
e^{it[x_2^q+e^{-1/|x_1|^p}]}
\varphi(x_1,x_2)
dx_1dx_2,\\
&J^{(2)}(t)
=
\int_0^{\infty}\int_{\frac{1}{(\log t)^{1/p}}}^{\infty}
e^{it[x_2^q+e^{-1/|x_1|^p}]}
\varphi(x_1,x_2)
dx_1dx_2.
\end{split}
\label{eqn:5.10}
\end{equation}
The behaviors of $J^{(1)}(t)$ and  $J^{(2)}(t)$ as $t\to\infty$
are seen as follows.

\begin{lmm}
\begin{enumerate}
\item[{\rm (i)}] 
$$
\lim_{t\to\infty}
t^{1/q}(\log t)^{1/p} \cdot J^{(1)}(t)
=\Gamma(1/q+1)\cdot e^{\frac{\pi}{2q}i}
\cdot\varphi(0,0).
$$
\item[{\rm (ii)}]
There exist positive numbers $C$ and $t_0$ independent of
$t$ such that 
$$
|J^{(2)}(t)|\leq\frac{C}{t^{1/q}(\log t)^{1/p+1}}
\mbox{ \quad for $t\geq t_0$.}
$$
\end{enumerate}
\end{lmm}
From (\ref{eqn:5.9}), 
the above lemma easily implies the equation (\ref{eqn:5.8}).

\section{The proof of Lemma 3.1}

Let us prove Lemma~3.1. 
Let $\alpha$ be a smooth function defined on $\mathbb R$ 
satisfying 
that $\alpha(x)=1$ for $|x|\leq 1$ and 
$\alpha(x)=0$ for $|x|\geq 2$, and let $\beta(x):=1-\alpha(x)$. 

{\bf (i)}.\quad 
Let $P(x_1,x_2):=e^{x_2^q}\varphi(x_1,x_2)$.
It is easy to see 
\begin{equation*}
P(x_1,x_2)
=P(x_1,0)+x_2\int_0^1
\frac{\partial P}{\partial x_2}(x_1,sx_2)ds.
\end{equation*}
By using the functions 
$\alpha$ and $\beta$,
we have
\begin{equation*}
\begin{split}
P(x_1,x_2)
&=
P(x_1,0)-\beta(x_2)P(x_1,0)+x_2 e^{x_2^q}R(x_1,x_2)\\
&=
\alpha(x_2)P(x_1,0)+x_2 e^{x_2^q}R(x_1,x_2),
\end{split}
\end{equation*}
where
$$
R(x_1,x_2)=
e^{-x_2^q}\left(
\frac{1}{x_2}\beta(x_2)P(x_1,0)
+
\int_0^1
\frac{\partial P}{\partial x_2}(x_1,sx_2)ds
\right).
$$
Noticing that the supports of $P(x_1,x_2)$ and 
$\alpha(x_2)P(x_1,0)$ are compact, 
we see that $R$ is a smooth function on 
${\mathbb R}^2$ 
with a compact support.  
Since $\varphi(x_1,x_2)=e^{-x_2^q}P(x_1,x_2)$, 
$\varphi$ can be expressed as 
\begin{equation}
\varphi(x_1,x_2)
=
e^{-x_2^q} \varphi(x_1,0)-
e^{-x_2^q} \beta(x_2)\varphi(x_1,0)+x_2 R(x_1,x_2).
\label{eqn:5.11}
\end{equation}
By substituting (\ref{eqn:5.11}) into (\ref{eqn:5.10}) and 
applying Fubini's theorem, 
the integral $J^{(1)}(t)$ can be expressed as
\begin{equation*}
J^{(1)}(t)=
K^{(1)}(t)-K^{(2)}(t)+K^{(3)}(t),
\end{equation*}
with
\begin{equation}
\begin{split}
&K^{(1)}(t)
=
L^{(1)}(t;\varphi(\cdot,0))
\cdot
\int_0^{\infty}
e^{-[1-it]x_2^q}
dx_2, \\
&K^{(2)}(t)
=
L^{(1)}(t;\varphi(\cdot,0))
\cdot
\int_0^{\infty}
e^{-[1-it]x_2^q}
\beta(x_2)dx_2, \\
&K^{(3)}(t)
=\int_0^{\infty}
e^{itx_2^q}
L^{(1)}(t;R(\cdot,x_2))
x_2dx_2.
\end{split}
\label{eqn:5.12}
\end{equation}
Now, let us investigate the behaviors of the 
above three functions as $t\to\infty$.

\vspace{.5 em}

(Behavior of $K^{(1)}(t)$.)\quad 

First, let us consider the integral
$K^{(1)}(t)$. 
Setting $z=[1-it]^{1/q}x_2$ and noting
that the rapid decay of $e^{-z^q}$ allows 
us to replace the contour $[1-it]^{1/q}\cdot[0,\infty)$
by $[0,\infty)$, we see that 
$$
\int_0^{\infty}
e^{-[1-it]x_2^q}dx_2
=\frac{1}{(1-it)^{1/q}}
\cdot\int_0^{\infty}e^{-z^q}dz
=\frac{\Gamma(1/q+1)}{(1-it)^{1/q}}
=\frac{1}{t^{1/q}}
\frac{\Gamma(1/q+1)}{(1/t-i)^{1/q}}.
$$
On the other hand, 
Lemma~2.1 (i) implies 
\begin{equation*}
\lim_{t\to\infty}
(\log t)^{1/p} \cdot 
L^{(1)}(t;\varphi(\cdot,0))=
\varphi(0,0).
\end{equation*}
Applying the above equalities to (\ref{eqn:5.12}), 
we have 
\begin{equation*}
\begin{split}
\lim_{t\to\infty}
t^{1/q}(\log t)^{1/p}\cdot K^{(1)}(t)
=\Gamma(1/q+1) e^{\frac{\pi}{2q}i}\cdot \varphi(0,0).
\end{split}
\end{equation*}

\vspace{.5 em}
(Estimate of $K^{(2)}(t)$.)

Let $N$ be an arbitrary natural number. 
Applying $N$-times integrations by parts,
we have 
\begin{equation*}
\int_0^{\infty}
e^{-[1-it]x_2^q}
\beta(x_2)dx_2=
\frac{1}{q^N[1-it]^N}
\int_0^{\infty}
e^{-[1-it]x_2^q}
\left(
\frac{\partial}{\partial x_2}\cdot
\frac{1}{x_2^{q-1}}
\right)^N
\beta(x_2)dx_2.
\end{equation*}
A simple computation implies that there exist
positive numbers $t_N$ and $C_N$ such that
$$
\left|
\int_0^{\infty}
e^{-[1-it]x_2^q}
\beta(x_2)dx_2
\right|
\leq \frac{C_N}{t^N} 
\quad\quad \mbox{for $t\geq t_N$.}
$$
Therefore, 
from (\ref{eqn:5.12}) and Lemma~2.1 (i), 
there exist positive numbers $\tilde{t}_N$ and $\tilde{C}_N$ such that
\begin{equation*}
|K^{(2)}(t)|
\leq \frac{\tilde{C}_N}{t^N(\log t)^{1/p}}
\quad\quad \mbox{for $t\geq \tilde{t}_N$.}
\end{equation*}

\vspace{.5 em}

(Estimate of $K^{(3)}(t)$.)

For the proof of (i), it suffices to show that 
the integral $K^{(3)}(t)$ is dominated by 
$Ct^{-2/q}(\log t)^{-1/p}$ for large $t$. 

The integral $K^{(3)}(t)$ can be written as follows:
\begin{equation} 
K^{(3)}(t)=H^{(1)}(t)+H^{(2)}(t),
\label{eqn:5.13}
\end{equation}
with
\begin{equation*}
\begin{split}
&H^{(1)}(t)
=
\int_0^{\infty}
e^{itx_2^q}
L^{(1)}(t;R(\cdot,x_2))\alpha(t^{1/q}x_2)
x_2 dx_2,\\
&
H^{(2)}(t)
=
\int_0^{\infty}
e^{itx_2^q}
L^{(1)}(t;R(\cdot,x_2))\beta(t^{1/q}x_2)
x_2 dx_2,
\end{split}
\end{equation*}
where the functions $\alpha$ and $\beta$ are as 
in the beginning of this subsection. 

Let us investigate 
the behaviors of the functions $H^{(1)}(t)$ and $H^{(2)}(t)$
as $t\to\infty$.

\vspace{.5 em}

(Behavior of $H^{(1)}(t)$.) \quad 

Exchanging the integral variable $x_2$ by $u_2$: 
$u_2=t^{1/q}x_2$, we have
\begin{equation}
\begin{split}
H^{(1)}(t)= 
\frac{1}{t^{2/q}}\int_0^2
e^{iu_2^q}
L^{(1)}(t;R(\cdot,\frac{u_2}{t^{1/q}}))\alpha(u_2)
u_2du_2.
\end{split}
\label{eqn:5.15}
\end{equation}
In order to investigate the behavior of
$L^{(1)}(t;R(\cdot,\frac{u_2}{t^{1/q}}))$ 
as $t\to\infty$, 
consider the following inequality:
\begin{equation}
\begin{split}
&\left|
(\log t)^{1/p}\cdot L^{(1)}(t;R(\cdot,\frac{u_2}{t^{1/q}}))-R(0,0)
\right|\\
&\quad \leq (\log t)^{1/p}
\left|
L^{(1)}(t;R(\cdot,\frac{u_2}{t^{1/q}}))-
L^{(1)}(t;R(\cdot,0))
\right|\\
&\quad\quad\quad +
\left|
(\log t)^{1/p}L^{(1)}(t;R(\cdot,0))-R(0,0)
\right|.
\end{split}
\label{eqn:5.16}
\end{equation}
The first term in the right hand side of (\ref{eqn:5.16}) 
is dominated by 
$$
(\log t)^{1/p}\int_0^{\frac{1}{(\log t)^{1/p}}}
\left|
R(x_1,\frac{u_2}{t^{1/q}})-R(x_1,0)
\right|
dx_1 \quad
\mbox{ for $u_2\in [0,2]$}.
$$
The uniform of the continuity of the function
$R$ implies that the above integral tends to zero 
as $t\to\infty$.
Moreover, from Lemma 2.1 (i), 
the second term 
in the right hand side of (\ref{eqn:5.16})
tends to zero as $t\to\infty$. 
Therefore, we have 
\begin{equation}
\lim_{t\to\infty}
 (\log t)^{1/p} L^{(1)}(t;R(\cdot,\frac{u_2}{t^{1/q}}))=R(0,0)
\mbox{ \quad for $u_2\in[0,2]$}.
\label{eqn:5.16a}
\end{equation}
Note that the limit in (\ref{eqn:5.16a}) 
is uniform with respect to 
$u_2\in[0,2]$. 
Applying the equality (\ref{eqn:5.16a}) to (\ref{eqn:5.15}), 
we can easily get
\begin{equation}
\lim_{t\to \infty}
t^{2/q}(\log t)^{1/p}\cdot
H^{(1)}(t)=R(0,0)\cdot \int_0^{2}
e^{iu_2^q}\alpha(u_2)u_2 du_2.
\label{eqn:5.17}
\end{equation}


\vspace{.5 em}

(Estimate for $H^{(2)}(t)$.) \quad 

By applying two-times integrations by parts 
to the integral $H^{(2)}(t)$, 
$H^{(2)}(t)$ can be written as  
\begin{equation*}
H^{(2)}(t)=
\left(\frac{-1}{qit}\right)^2
\int_0^{\infty} 
e^{itx_2^q}
L^{(1)}(t;F(\cdot,x_2;t))dx_2,
\end{equation*}
where
\begin{equation*}
\begin{split}
F(x_1,x_2;t)= \frac{\partial}{\partial x_2}
\left(\frac{1}{x_2^{q-1}}\cdot\frac{\partial}{\partial x_2}
\left(\frac{1}{x_2^{q-1}}\cdot x_2
R\left(
x_1,x_2
\right)
\beta(t^{1/q}x_2)
\right)\right).
\end{split}
\end{equation*}
A simple computation shows that
there is a positive constant $C$ independent of 
$x_1$ and $t$ such that
\begin{equation}
|F(x_1,x_2;t)|\leq 
\frac{C}{x_2^{2q-1}} \quad \mbox{ for $x_2>0$.}
\label{eqn:5.18}
\end{equation}
Note that $t^{1/q}$ is dominated by $2/x_2$ 
when $t^{1/q}x_2$ is contained in the support of $\beta'$.  
Moreover, there exist positive numbers $t_0$, $C$ such that
\begin{equation}
\begin{split}
&|L^{(1)}(t;F(\cdot,x_2;t))|\leq
\left|
\int_0^{\frac{1}{(\log t)^{1/p}}}
e^{it e^{-1/x_1^p}}F(x_1,x_2;t)dx_1 \right|\\
&\quad\quad 
\leq \int_0^{\frac{1}{(\log t)^{1/p}}} 
|F(x_1,x_2;t)|dx_1 
\leq \frac{C}{x_2^{2q-1}(\log t)^{1/p}}
\quad
\mbox{ for $x_2>0$, $t \geq t_0$}.
\end{split}
\label{eqn:5.19}
\end{equation}
By noticing that $2q-1\geq 3(>1)$ and that 
the support of $F(u_1,\cdot;t)$
is contained in $(t^{-1/q},\infty)$, 
the inequalities in (\ref{eqn:5.19}) imply that
\begin{equation}
\begin{split}
|H^{(2)}(t)|\leq
&\frac{C}{t^2(\log t)^{1/p}}\cdot
\int_{t^{-1/q}}^{\infty}
\frac{1}{x_2^{2q-1}} dx_2 \\
\leq&
\frac{C}{t^2 (\log t)^{1/p}\cdot t^{-2+2/q}}\leq 
\frac{C}{t^{2/q}(\log t)^{1/p}}
\quad
\mbox{ for $t\geq t_0$}.
\end{split}
\label{eqn:5.20}
\end{equation}
We remark that the function $F$ with
the estimate (\ref{eqn:5.18}) 
was obtained by applying {\it integration by parts} 
and it played an important role in the estimate (\ref{eqn:5.20})
of the integral $H^{(2)}(t)$ 
(see also the first remark in the end of Section 2).

\vspace{.5 em}

Putting (\ref{eqn:5.13}), (\ref{eqn:5.17}), (\ref{eqn:5.20}) together, 
we can get the desired estimate:
\begin{equation*}
|K^{(3)}(t)|\leq \frac{C}{t^{2/q}(\log t)^{1/p}}
\quad
\mbox{ for $t\geq t_0$.}
\end{equation*}

\vspace{.5 em}

{\bf (ii)}.\quad 
By using the functions $\alpha$ and $\beta$ in the beginning of 
Section~4, 
the integral $J^{(2)}(t)$ can be devided as 
\begin{equation} 
J^{(2)}(t)=N^{(1)}(t)+N^{(2)}(t),
\label{eqn:5.21}
\end{equation}
with
\begin{equation}
\begin{split}
&N^{(1)}(t)
=
\int_0^{\infty}
e^{itx_2^q}
L^{(2)}(t;\varphi(\cdot,x_2))\alpha(t^{1/q}x_2)
dx_2,\\
&
N^{(2)}(t)
=
\int_0^{\infty}
e^{itx_2^q}
L^{(2)}(t;\varphi(\cdot,x_2))\beta(t^{1/q}x_2)
dx_2.
\end{split}
\label{eqn:5.22}
\end{equation}
To obtain the estimate in (ii) in Lemma 3.1, 
we give appropriate estimates for
$N^{(1)}(t)$ and $N^{(2)}(t)$.

\vspace{.5 em}

(Estimate for $N^{(1)}(t)$.) \quad 

Exchanging the integral variable $x_2$ by $u_2$:
$x_2=u_2/t^{1/q}$, we have
\begin{equation}
N^{(1)}(t)
=
\frac{1}{t^{1/q}}
\int_0^{2}
e^{iu_2^q}
L^{(2)}(t;\varphi(\cdot,\frac{u_2}{t^{1/q}}))\alpha(u_2)
du_2.
\label{eqn:5.23}
\end{equation}

From Lemma~2.1 (ii), 
there exist positive numbers $t_0$, $C$ 
independent of $t$ and $x_2$ such that
\begin{equation*}
|L^{(2)}(t;\varphi(\cdot,x_2))|
\leq \frac{|\varphi(0,x_2)|+C}{(\log t)^{1/p+1}}
\leq \frac{C}{(\log t)^{1/p+1}} 
\quad \mbox{for $t\geq t_0$, $x_2>0$}.
\end{equation*}
Therefore, applying the above estimate to (\ref{eqn:5.23}), 
we have
\begin{equation}
|N^{(1)}(t)|\leq \frac{C}{t^{1/q}(\log t)^{1/p+1}}
\quad \mbox{for $t\geq t_0$}.
\label{eqn:5.24}
\end{equation}

\vspace{.5 em}

(Estimate for $N^{(2)}(t)$.) \quad 

By appying integration by parts to (\ref{eqn:5.22}),
the integral $N^{(2)}(t)$ can be written as  
\begin{equation*}
N^{(2)}(t)
=
\frac{-1}{qit}
\cdot
\int_0^{\infty}
e^{itx_2^q}
L^{(2)}(t;G(\cdot,x_2;t))
dx_2,
\end{equation*}
where 
\begin{equation*}
G(x_1,x_2;t)= 
\frac{\partial}{\partial x_2}
\left(\frac{1}{x_2^{q-1}}
\varphi\left(
x_1,x_2
\right)
\beta(t^{1/q}x_2)
\right).
\end{equation*}
Let 
$\tilde{G}(x_1,x_2;t)=
x_2^{q}G(x_1,x_2;t)$. 
A simple computation 
shows that $\tilde{G}$ is bounded on 
$[0,\infty)^3$.
Note that $t^{1/q}$ is dominated by $2/x_2$ 
when $t^{1/q}x_2$ is contained in the support of $\beta'$.
From Lemma~2.1 (ii), 
\begin{equation*}
\lim_{t\to\infty}
(\log t)^{1/p+1}\cdot L^{(2)}(t;\tilde{G}(\cdot,x_2;t))=
ie^i\cdot \tilde{G}(0,x_2;t).
\end{equation*}
The boundedness of $\tilde{G}$ implies
that there exist positive numbers $t_0$ and  $C$
independent of $t,x_2$ such that 
\begin{equation}
|L^{(2)}(t;\tilde{G}(\cdot,x_2;t))|
\leq \frac{C}{(\log t)^{1/p+1}}
\quad
\mbox{ for $t\geq t_0$, $x_2 > 0$.}
\label{eqn:5.25}
\end{equation}
Therefore, by noticing that 
the support of $F(u_1,\cdot;t)$
is contained in $(t^{-1/q},\infty)$, 
(\ref{eqn:5.25}) implies that
there exist positive numbers $t_0$ and $C$ 
independent of $t$ such that
\begin{equation}
\begin{split}
|N^{(2)}(t)|
&=
\frac{1}{qt}
\left|
\int_0^{\infty}
e^{it x_2^q}
L^{(2)}(t;G(\cdot,x_2;t))
dx_2
\right| \\
&=
\frac{1}{qt}
\left|
\int_0^{\infty}
e^{it x_2^q}
\frac{1}{x_2^q}
L^{(2)}(t;\tilde{G}(\cdot,x_2;t))
dx_2
\right| \\
&\leq
\frac{C}{t}\cdot\int_{t^{-1/q}}^{\infty}
\frac{1}{x_2^q}\frac{1}{(\log t)^{1/p+1}}dx_2 \\
&\leq 
\frac{C}{t} \cdot \frac{1}{t^{1/q-1}}\cdot \frac{1}{(\log t)^{1/q+1}}
=\frac{C}{t^{1/q}(\log t)^{1/p+1}}
\quad
\mbox{ for $t\geq t_0$.}
\end{split}
\label{eqn:5.26}
\end{equation}

Putting (\ref{eqn:5.21}), (\ref{eqn:5.24}), (\ref{eqn:5.26}) together, 
we can get (ii) in Lemma~3.1.




%
%

\end{document}